\documentclass[review]{elsarticle}

\usepackage{hyperref}
\usepackage{amsmath,amsthm}

\newtheorem{thm}{Theorem}[section]
\newtheorem{lemma}{Lemma}[section]
\numberwithin{equation}{section}

\journal{J. Math. Anal. Appl.}

\begin{document}

\begin{frontmatter}

\title{The convergence rate of  of multivariate operators on simplex in Orlicz space\tnoteref{mytitlenote}}
\tnotetext[mytitlenote]{This work was supported by the Ningxia Science and Technology Department [grant numbers 2019BEB04003]; and Ningxia Education Department [grant numbers NXYLXK2019A8]}

\author[mymainaddress]{Wan Ma}
\author[mymainaddress]{Lihong Chang}
\author[mymainaddress]{Yongxia Qiang\corref{mycorrespondingauthor}}
\cortext[mycorrespondingauthor]{Corresponding author}
\ead{415884392@qq.com}

\address[mymainaddress]{School of Mathematics and Computer Science, Ningxia Normal University, Guyuan City, Ningxia 756000,
                         People's Republic of China}

\begin{abstract}
  The approximation of functions in Orlicz space by multivariate operators on simplex is considered. The convergence rate is given by using modulus of smoothness.
\end{abstract}


\end{frontmatter}


\section{Introduction}

Let $\Phi (u)$ be a N-function, $\Psi$ be the complementary function of $\Phi$. We will say that $\Phi$ satisfies the $\Delta_2$-condition if $\Phi(2u)\leq c\Phi(u)$ for any $ u\geq u_0\geq 0$ with some constant $c$ independent of  $u.$

For $\triangle = \{x=(x_1,x_2)\in \mathrm{R^2}: x_1 + x_2\leq 1, x_1,x_2\geq 0\},$ the Orlicz space $ L_{\Phi}^{\ast}(\triangle) $ corresponding to the function  $\Phi$ consists of all Lebesgue-measurable functions $f(x)$ on $\triangle$ such that integral  $ \int_\triangle f(x)g(x)\mbox{d}x $ is finite for
any measurable functions $g(x)$ with $ \int_\triangle \Psi(g(x))\mbox{d}x < \infty. $

It is well-known that the space  $ L_{\Phi}^{\ast}(\triangle) $ becomes a complete normed space with Orlicz norm

$$ \|f\|_\Phi = \|f\|_{[\Phi,\triangle]} = \sup\left\{\left|\int_\triangle f(x)g(x)\mbox{d}x\right|:  \int_\triangle \Psi(g(x))\mbox{d}x \leq 1 \right\}. $$
It can be proved that

 $$ \|f\|_{\Phi} =  \inf_{\alpha>0}\frac{1}{\alpha}\left\{1+\int_\triangle \Phi(\alpha f(x))\mbox{d}x\right\}. $$
See \cite{ww83} for the above. For $ f\in L_{\Phi}^{\ast}(\triangle), $ we first extend $f(x)$ from $\triangle$  to $ D = [0,1]\times[0,1] $ according to $ f(x) = f(1-x) $, and then extend  $f(x)$ to $\mathrm{R^2}$ with period 1.   The nonnegative function

$$ \Omega_{\mathrm{R^2}}^{2}(f,r)_{\Phi} = \sup\{\omega _{h}^{2}(f,r)_{\Phi}:h = (h_1, h_2)\in \mathrm{R^2}, |h|=1\} $$

\noindent of the variable $r \geq 0$ will be called the 2-th order modulus of continuity of the function $ f\in L_{\Phi}^{\ast}(\triangle) $ in the Orlicz norm $\Phi.$ Here,  $ |h| = \sqrt{h_{1}^2+h_{2}^2}, $ and

$$ \omega _{h}^{2}(f,r)_{\Phi} = \sup_{|t|\leq r}\|f(x+th)+f(x-th)-2f(x)\|_{\Phi} $$
\noindent is the 2-th order modulus of continuity in the direction  $h$ of the function $ f. $

For any Lebesgue-measurable function $f(x)$ on $\triangle,$ the functional

\begin{equation}\label{eq:kn}
 K_{n}(f;x)=\sum_{k_1 =0}^{\infty}\sum_{k_2 =0}^{\infty} {\tilde{p}_{n,k_1,k_2}(x)} c_{n,k_1,k_2}\int_{\triangle _{k_1,k_2}}f(u)\mbox{d}u
\end{equation}

\noindent is called Meyer-K\"{o}nig-Zeller-$\mathrm{Kantorovi\check{c}}$ operator on $\triangle$ \cite{x94}; the functional

\begin{equation}\label{eq:kns} K_{n,s}(f;x)=\sum_{k+l\leq n}{b_{n,k,l,s}(x)} (n+2)^{2} \int_{I_{n,k,l}}f(u)\mbox{d}u\end{equation}

\noindent is called Stancu-$\mathrm{Kantorovi\check{c}}$ operator on $\triangle$\cite{xyc93}, where $ x\in \triangle , $  $ n\in \mathrm{Z^+}, $ $ k,s,l $ are nonnegative integers, $ 0\leq s <\frac{n}{2},$ and

$$ {\tilde{p}_{n,k_1,k_2}(x)} = \frac{(n+k_1+k_2)!}{n!k_{1}!k_{2}!}x_{1}^{k_1}x_{2}^{k_2}(1-x_1-x_2)^{n+1}, $$

$$ c_{n,k_1,k_2} = \frac{(n+k_{1}+k_{2})^{2}(n+k_{1}+k_{2}+1)^{2}}{(n+k_{1})(n+k_{2})}, $$

$$ \triangle _{k_1,k_2} = \left[ \frac{k_1}{n+k_{1}+k_{2}},\frac{k_{1}+1}{n+k_{1}+k_{2}+1}\right]\times \left[ \frac{k_2}{n+k_{1}+k_{2}},
\frac{k_{2}+1}{n+k_{1}+k_{2}+1}\right],$$

$$ I_{n,k,l} = \left[ \frac{k}{n+2}, \frac{k+1}{n+2}\right]\times\left[ \frac{l}{n+2}, \frac{l+1}{n+2}\right], $$

$$ {p_{n,k,l}(x)} = \frac{n!}{k!l!(n-k-l)!}x_{1}^{k}x_{2}^{l}(1-x_1-x_2)^{n-k-l}, $$

$$ {b_{n,k,l,s}(x)} = \begin{cases}
    (1-x_1-x_2)p_{n-s,k,l}(x), \quad k+l\leq n-s, 0\leq k,l <s; \\
    (1-x_1-x_2)p_{n-s,k,l}(x)+x_1 p_{n-s,k-s,l}(x), \\ \qquad \qquad \qquad \qquad \qquad \quad \;\;\; k+l\leq n-s, s\leq k, 0\leq l <s; \\
   (1-x_1-x_2)p_{n-s,k,l}(x)+x_2 p_{n-s,k,l-s}(x),  \\  \qquad \qquad \qquad \qquad \qquad \quad \;\;\; k+l\leq n-s, s\leq l, 0\leq k <s; \\
   (1-x_1-x_2)p_{n-s,k,l}(x)+x_1 p_{n-s,k-s,l}(x)+x_2 p_{n-s,k,l-s}(x,y),\\ \qquad \qquad \qquad \qquad \qquad \qquad  k+l\leq n-s,s\leq k, s\leq l; \\
   x_1 p_{n-s,k-s,l}(x), \quad \qquad \qquad n-s < k+l\leq n, s\leq k, 0\leq l <s; \\
   x_2 p_{n-s,k,l-s}(x), \quad \qquad \qquad n-s < k+l\leq n, 0\leq k<s, s\leq l; \\
   x_1 p_{n-s,k-s,l}(x)+x_2 p_{n-s,k,l-s}(x), \\ \qquad \qquad \qquad \qquad \qquad \qquad\;  n-s < k+l\leq n, s\leq k, s\leq l.
 \end{cases} $$

Denote $C$ a constant independent of $ f,n,s, $ and its value can be different in different positions. The convergence rate of the operators (\ref{eq:kn}) and (\ref{eq:kns}) in space $L_{p}$ has been studied (see \cite{x94, xyc93}). This paper intends to investigate their convergence in space $ L_{\Phi}^{\ast}(\triangle),$ and the main results are as follows.

\begin{thm}\label{thm:kn}For $ f\in L_{\Phi}^{\ast}(\triangle), $  if a N-function $\Phi (u)$ satisfies the $\Delta_2$-condition, then

$$ \| K_{n}(f)-f \|_{\Phi}\leq C \left(\frac{1}{n}\| f \|_{\Phi}+\Omega_{\mathrm{R^2}}^{2}\left(f,\sqrt{\frac{1}{n}}\right)_{\Phi}\right).  $$ \end{thm}

\begin{thm}\label{thm:kns}For $ f\in L_{\Phi}^{\ast}(\triangle), $  if a N-function $\Phi (u)$ satisfies the $\Delta_2$-condition, then

$$ \| K_{n,s}(f)-f \|_{\Phi}\leq C \left(\frac{1}{n}\| f \|_{\Phi}+\Omega_{\mathrm{R^2}}^{2}\left(f,\sqrt{\frac{1}{n}}\right)_{\Phi}\right).  $$ \end{thm}

\section{Lemmas}

\begin{lemma}\label{lem:knb}  $ K_n $ is a bounded linear operator, and $\| K_n \|_\Phi \leq 2. $ \end{lemma}
\begin{proof} The linearity of $ K_n $  is obvious. The following proves $\| K_n \|_\Phi \leq 2. $ After calculation, we can get

  $$ mes\triangle_{k_1, k_2} = \frac{(n+k_{1})(n+k_{2})}{(n+k_{1}+k_{2})^{2}(n+k_{1}+k_{2}+1)^{2}}, $$
  $$ \int_{\triangle}\tilde{p}_{n,k_1,k_2}(x) \mbox{d}x = \frac{n+1}{(n+k_{1}+k_{2}+3)(n+k_{1}+k_{2}+2)(n+k_{1}+k_{2}+1)}. $$

\noindent By using the \emph{Lemma 1} in \cite{x94}, Jensen inequality of convex function and the\emph{Theorem 1.4} in \cite{ww83}, we obtain

$$ \aligned \| K_{n}(f)\|_{\Phi} & = \inf_{\alpha > 0}\frac{1}{\alpha}\left\{1+\int_{\triangle}\Phi\left(\alpha K_{n}(f;x)\right)\mbox{d}x\right\}\\
& = \inf_{\alpha > 0}\frac{1}{\alpha}\left\{1+\int_{\triangle}\Phi\left(\alpha \sum_{k_1 =0}^{\infty}\sum_{k_2 =0}^{\infty}\tilde{p}_{n,k_1,k_2}(x) c_{n,k_1,k_2}\int_{\triangle _{k_1,k_2}}f(u)\mbox{d}u\right)\mbox{d}x\right\}\\
& \leq \inf_{\alpha > 0}\frac{1}{\alpha}\left\{1+\sum_{k_1 =0}^{\infty}\sum_{k_2 =0}^{\infty}\int_{\triangle}\tilde{p}_{n,k_1,k_2}(x)\mbox{d}x c_{n,k_1,k_2}\int_{\triangle _{k_1,k_2}}\Phi\left(\alpha f(u)\right)\mbox{d}u\right\}\\
& \leq \inf_{\alpha > 0}\frac{1}{\alpha}\{1+2\sum_{k_1 =0}^{\infty}\sum_{k_2 =0}^{\infty}\int_{\triangle _{k_1,k_2}}\Phi\left(\alpha f(u)\right)\mbox{d}u\}\\
& \leq \inf_{\alpha > 0}\frac{1}{\alpha}\left\{1+\int_{\triangle}\Phi\left(2\alpha f(u)\right)\mbox{d}u_{1}\mbox{d}u_{2}\right\}\\
& = \| 2f\|_{\Phi} \\ & =  2\|f\|_{\Phi}. \endaligned $$
\noindent Namely  $$ \|K_{n}\|_{\Phi}\leq 2.  \qedhere $$  \end{proof}

\begin{lemma}\label{lem:knsb}  $K_{n, s}$ is a bounded linear operator, and $ \| K_{n, s} \|_\Phi \leq 12. $ \end{lemma}
\indent \begin{proof}  The linearity of $ K_{n, s} $  is obvious. The following proves $ \| K_{n, s} \|_\Phi \leq 12. $  After calculation, we can get

$$ mesI_{n, k, l} = mesI_{n, k+s, l} = mesI_{n, k, l+s} = \frac{1}{(n+2)^2} , $$
$$ \int_{\triangle}p_{n-s, k, l}(x)\mbox{d}x = \frac{1}{(n-s+2)(n-s+1)}. $$

By using the \emph{Lemma 2.1} in \cite{xyc93}, Jensen inequality of convex function and the \emph{Theorem 1.4} in \cite{ww83}, we obtain

$$ \aligned \| K_{n, s}(f)\|_{\Phi} & \leq \inf_{\alpha > 0}\frac{1}{\alpha}\left\{1+\sum_{k+l\leq n-s}\int_{\triangle}p_{n-s, k, l}(x)\mbox{d}x\left(\int_{I_{n, k, l}}+\int_{I_{n, k+s, l}}\int_{I_{n, k, l+s}}\right)\Phi\left(\alpha f(u)\right)\mbox{d}u\right\}\\
& \leq \inf_{\alpha > 0}\frac{1}{\alpha}\left\{1+\frac{3(n+2)^{2}}{(n-s+2)(n-s+1)}\int_{\triangle}\Phi(\alpha f(u))\mbox{d}u\right\}.\\
& \leq \inf_{\alpha > 0}\frac{1}{\alpha}\left\{1+\int_{\triangle}\Phi\left(\alpha 12f(u)\right)\mbox{d}u\right\}\\
& = \| 12f\|_{\Phi} \\ & =  12\|f\|_{\Phi}.\endaligned $$

\noindent Namely

  $$ \|K_{n, s}\|_{\Phi}\leq 12.  \qedhere $$  \end{proof}

\begin{lemma}\label{lem:kne} The following holds for (\ref{eq:kn}).

  $$ K_{n}(1; x) = 1, \quad K_{n}\left((u_i-x_i)^{i}; x\right) \leq \frac{C}{n}, \quad i = 1, 2.$$
\end{lemma}

\begin{lemma}\label{lem:knse}The following holds for (\ref{eq:kns}).

  $$ K_{n,s}(1; x) = 1, \quad K_{n,s}\left((u_i-x_i)^{i}; x\right) \leq \frac{C}{n}, \quad i = 1, 2.$$
\end{lemma}

The proof of the Lemma \ref{lem:kne} and Lemma \ref{lem:knse}  can be obtained from the \emph{Lemma 3} in \cite{x94} and the \emph{Lemma 2.1} in \cite{xyc93}.

\begin{lemma}\label{lem:steklov} If we denote $ f_{r} $ the Steklov mean function for $ f\in L_{\Phi}^{\ast}(\triangle), $  i.e.

$$  f_{r}(x)=\frac{1}{r^4}\int_{[-r/2,r/2]^{4}}f(x+u+v)\mbox{d}s\mbox{d}t, $$

\noindent then

\begin{equation}\label{eq:steklov1}\left\| {f_r } \right\|_\Phi   \le C\left\| f \right\|_\Phi,\end{equation}
\begin{equation}\label{eq:steklov2}\left\| {f - f_r } \right\|_\Phi   \le C\Omega _{R^2 }^2 \left( {f,r} \right)_\Phi  ,\end{equation}
\begin{equation}\label{eq:steklov3}\left\| {\frac{{\partial f_r }}{{\partial x_1 }}} \right\|_\Phi   \le C\left( {\left\| {f_r } \right\|_\Phi   + \left\| {\frac{{\partial ^2 f_r }}{{\partial x_1^2 }}} \right\|_\Phi  } \right),\end{equation}
\begin{equation}\label{eq:steklov4}\left\| {\frac{{\partial f_r }}{{\partial x_2 }}} \right\|_\Phi   \le C\left( {\left\| {f_r } \right\|_\Phi   + \left\| {\frac{{\partial ^2 f_r }}{{\partial x_2^2 }}} \right\|_\Phi  } \right),\end{equation}
\begin{equation}\label{eq:steklov5}\left\| {\frac{{\partial ^2 f_r }}{{\partial x_1^2 }}} \right\|_\Phi   + \left\| {\frac{{\partial ^2 f_r }}{{\partial x_2^2 }}} \right\|_\Phi   + \left\| {\frac{{\partial ^2 f_r }}{{\partial x_1 \partial x_2 }}} \right\|_\Phi   \le \frac{C}{{r^2 }}\Omega _{R^2 }^2 \left( {f,r} \right)_\Phi.\end{equation}  \end{lemma}

(\ref{eq:steklov1}), (\ref{eq:steklov2}), (\ref{eq:steklov5}) can be directly verified, and the proof of (\ref{eq:steklov3}),(\ref{eq:steklov4}) is similar to that of the \emph{Lemma 1a} in \cite{r84}. If N-function $\Phi$ satisfies the $\Delta_2$-condition, then $ L_{\Phi}^{\ast} $ is separable. This leads to the following conclusion\cite{x93}.

\begin{lemma}\label{lem:de}If N-function $\Phi$ satisfies the $\Delta_2$-condition, then

  $$ \left\|\sup_{u_1\neq x_1}\frac{1}{u_{1}-x_{1}}\int_{x_1}^{u_1}\left|\frac{\partial ^{2}f_{r}(\xi, x_2)}{\partial \xi^2}\right|\mbox{d}\xi\right\|_\Phi\leq C \left\| {\frac{{\partial ^2 f_r }}{{\partial x_1^2 }}} \right\|_\Phi, $$
  $$ \left\|\sup_{u_2\neq x_2}\frac{1}{u_{2}-x_{2}}\int_{x_2}^{u_2}\left|\frac{\partial ^{2}f_{r}(x_1, \eta)}{\partial \eta^2}\right|\mbox{d}\eta\right\|_\Phi\leq C \left\| {\frac{{\partial ^2 f_r }}{{\partial x_2^2 }}} \right\|_\Phi, $$
  $$ \left\|\sup_{u_2\neq x_2}\frac{1}{u_{2}-x_{2}}\int_{x_{2}}^{u_2}\left(\sup_{u_1\neq x_1}\int_{x_1}^{u_1}\left|\frac{\partial ^{2}f_{r}(\xi, \eta)}{\partial \xi \partial \eta}\right|
\mbox{d}\xi\right)\mbox{d}\eta\right\|_\Phi\leq C \left\| {\frac{{\partial ^2 f_r }}{{\partial x_1\partial x_2 }}} \right\|_\Phi. $$
\end{lemma}

\section{Proof of the main results}

The proof of the Theorem \ref{thm:kn} and Theorem \ref{thm:kns} is similar, so only the Theorem \ref{thm:kn} is proved below.

\begin{proof}  Because

$$ \aligned  f_{r}(u)-f_{r}(x)  = (u_1-x_1)\frac{\partial f_{r}(x)}{\partial x_1} & +(u_2-x_2)\frac{\partial f_{r}(x)}{\partial x_2} +\int_{x_1}^{u_1}(u_1-\xi)\frac{\partial^{2} f_{r}(\xi ,x_2)}{\partial \xi^2}\mbox{d}\xi\\
{}\\
& + \int_{x_2}^{u_2}(u_2-\eta)\frac{\partial^{2} f_{r}(x_1,\eta)}{\partial \eta^2}\mbox{d}\eta+
\int_{x_1}^{u_1}\int_{x_2}^{u_2}\frac{\partial^{2} f_{r}(\xi, \eta)}{\partial \xi\partial \eta}\mbox{d}\eta\mbox{d}\xi,\endaligned $$

\noindent  so

$$ \aligned  \left|K_{n}(f_{r}; x)-f_{r}(x)\right| \leq & \left|K_{n}((u_{1}-x_{1}); x)\right|\left|\frac{\partial f_{r}(x)}{\partial x_1}\right|+\left|K_{n}((u_{2}-x_{2}); x)\right|\left|\frac{\partial f_{r}(x)}{\partial x_2}\right|+ \\
& \left|K_{n}((u_{1}-x_{1})^{2}; x)\right|\left(\sup_{u_1\neq x_1}\frac{1}{u_{1}-x_{1}}\int_{x_1}^{u_1}\left|\frac{\partial ^{2}f_{r}(\xi, x_2)}{\partial \xi^2}\right|\mbox{d}\xi\right)+\\
&\left|K_{n}((u_{2}-x_{2})^{2}; x)\right|\left(\sup_{u_2\neq x_2}\frac{1}{u_{2}-x_{2}}\int_{x_2}^{u_2}\left|\frac{\partial ^{2}f_{r}(x_1, \eta)}{\partial \eta^2}\right|\mbox{d}\eta\right)+\\
& \left|K_{n}(|u_{1}-x_{1}||u_{2}-x_{2}|; x)\right|\left(\sup_{u_2\neq x_2}\frac{1}{u_{2}-x_{2}}\int_{x_{2}}^{u_2}\left(\sup_{u_1\neq x_1}\int_{x_1}^{u_1}\left|\frac{\partial ^{2}f_{r}(\xi, \eta)}{\partial \xi \partial \eta}\right|
\mbox{d}\xi\right)\mbox{d}\eta\right). \endaligned $$

\noindent Noticing

 $$ |K_{n}(|u_{1}-x_{1}||u_{2}-x_{2}|; x)|\leq \frac{1}{2}|K_{n}((u_{1}-x_{1})^{2}; x)|+|K_{n}((u_{2}-x_{2})^{2}; x)|, $$

we continue the above estimation using the Lemma \ref{lem:kne}, Lemma \ref{lem:steklov} and Lemma \ref{lem:de}.

$$ \aligned \| K_{n}(f_{r})-f_{r}\|_{\Phi} & \leq \frac{C}{n}\left(\left\|\frac{\partial f_{r}}{\partial x_1}\right\|_{\Phi}+
  \left\|\frac{\partial f_{r}}{\partial x_2}\right\|_{\Phi}+\left\|\frac{\partial ^{2}f_{r}}{\partial x_1^2}\right\|_{\Phi}+
  \left\|\frac{\partial ^{2}f_{r}}{\partial x_2^2}\right\|_{\Phi}+
  \left\|\frac{\partial ^{2}f_{r}}{\partial x_1\partial x_2}\right\|_{\Phi}\right)\\
& \leq \frac{C}{n}\left(\left\|f\right\|_{\Phi}+\frac{1}{r^2}\Omega_{\mathrm{R}^2}^{2}(f,r)_{\Phi}\right).\endaligned $$

\noindent If $ r= \sqrt{\frac{1}{n}}, $  then

$$\|K_{n}(f_{r})-f_{r}\|_{\Phi}\leq \frac{C}{n}\left(\|f\|_{\Phi}+n\Omega_{\mathrm{R}^2}^{2}\left(f,\sqrt{\frac{1}{n}}\right)_{\Phi}\right). $$

\noindent For $ f\in L_{\Phi}^{\ast}(\triangle), $ using the Lemma \ref{lem:knb} and Lemma \ref{lem:steklov} we get

$$ \aligned \|K_{n}(f)-f\|_{\Phi} & \leq \| K_{n}(f)-K_{n}(f_{r})\|_{\Phi}+\|K_{n}(f_{r})-f_{r}\|_{\Phi}+\| f_{r}-f\|_{\Phi}\\
& \leq 3 \| f_{r}-f\|_{\Phi}+\|K_{n}(f_{r})-f_{r}\|_{\Phi}\\
& \leq C \Omega_{\mathrm{R}^2}^{2}(f,r)_{\Phi}+C\left(\frac{1}{n}\|f\|_{\Phi}+\Omega_{\mathrm{R}^2}^{2}\left(f,\sqrt{\frac{1}{n}}\right)_{\Phi}\right)\\
& \leq C \left(\frac{1}{n}\| f \|_{\Phi}+\Omega_{\mathrm{R}^2}^{2}(f,\sqrt{\frac{1}{n}})_{\Phi}\right).\hspace{3.5cm}\qedhere\endaligned $$  \end{proof}

\section{Remark}

If N-function $ \Phi (u)=u^{p} $ $ (1<p<\infty), $ then $ L_{\Phi}^{\ast}(\triangle)=L_p.$ Thus the corresponding conclusions in \cite{x94} and \cite{xyc93} can be obtained from the Theorem \ref{thm:kn} and Theorem \ref{thm:kns}.

\section*{References}

\bibliography{mybibfile}

\end{document}